\documentclass{amsart}      
\usepackage{amssymb, amsmath}
\newtoks\prt 
 
\newtheorem{thm}{Theorem}[section] 
\newtheorem{lemma}[thm]{Lemma} 
\newtheorem{prop}[thm]{Proposition} 
 
\newtheorem{example}[thm]{Example}

\theoremstyle{definition}

\def\eqn#1$$#2$${\begin{equation}\label#1#2\end{equation}}

\def\th{{\leavevmode\setbox1=\hbox{t}
  \hbox to \wd1{t\kern-0.6ex{\char039}\hss}}}
\def\dh{{\leavevmode\setbox1=\hbox{d}
  \hbox to 1.05\wd1{d\kern-0.4ex{\char039}\hss}}}
\def\=#1{\if #1u{\accent23u}\else 
\ifx #1d{\dh}\else 
\ifx #1t{\th}\else 
 {\accent20 #1}\fi\fi\fi} 
\def\'#1{\if #1i{\accent19\i}\else {\accent19 #1}\fi} 
\def\fra{\mathfrak{A}} 
\def\A{\mathcal A} 
 
\def\C{\mathcal C}

\def\F{\mathcal F}

\def\G{\mathcal G} 
\def\W{\mathcal W}

\def\L{\mathcal L} 
\def\U{\mathcal U} 
\def\V{\mathcal V} 
 
\def\H{\mathcal H} 
\def\M{\mathcal M}

\def\es{\mathbf S} 
 
\def\ce{\mathbb C}

\def\lin{Lindel\"of} 
\def\diam{\operatorname{diam}} 
\def\co{\operatorname{co}} 
\def\ep{\varepsilon} 
\def\K{\mathcal K} 
\def\en{\mathbb N} 
\def\er{\mathbb R}

\def\re{\operatorname{Re}} 
\def\im{\operatorname{Im}}

\def\r{|} 
 
\def\ov{\overline} 
\def \Ch {\operatorname{Ch}}

\def \Bos {\operatorname{Bos}} 
\def \Bas {\operatorname{Bas}} 
\def\Hs{\operatorname{Hs}}

\def \Bof {\operatorname{Bof}} 
\def \Baf {\operatorname{Baf}} 
\def\Hf{\operatorname{Hf}}

\def \ext {\operatorname{ext}}

\def\span{\operatorname{span}}

\def\wt{\widetilde}
\def \reg {\partial _{\kern1pt\text{reg}}}

\begin{document}

\title[Descriptive properties of elements of biduals of Banach spaces]{Descriptive properties of elements of biduals of Banach spaces}
\author{Pavel Ludv\'\i k and Ji\v r\'\i\  Spurn\'y}

\address{Department of Mathematical Analysis \\
Faculty of Mathematics and Physics\\ 
Charles University\\
Sokolovsk\'{a} 83, 186 \ 75\\Praha 8, Czech Republic}
\email{ludvik@karlin.mff.cuni.cz}
\email{spurny@karlin.mff.cuni.cz}

\subjclass[2010]{46B99,46A55,26A21}

\keywords{Baire and Borel functions, strongly affine functions, fragmentable functions, extreme points, $L_1$-preduals, Baire classes and intrinsic Baire classes of Banach spaces}

\thanks{The first author was supported by GA\v CR 401/09/H007 and SVV-2011-263316, the second author was supported in part by the grants 
GAAV IAA 100190901, GA\v CR  201/07/0388
and in part by the Research Project
MSM~0021620839 from the Czech Ministry of Education.
}

\begin{abstract}
If $E$ is a Banach space, any element $x^{**}$ in its bidual $E^{**}$ is an affine function on the dual unit ball $B_{E^*}$ that might possess variety of descriptive properties with respect to the weak* topology. We prove several results showing that descriptive properties of $x^{**}$ are quite often determined by the behaviour of $x^{**}$ on the set of extreme points of $B_{E^*}$, generalizing thus results of J.~Saint Raymond and F.~Jellett. We also prove several results on relation between Baire classes and intrinsic Baire classes of $L_1$-preduals which were introduced by S.A.~Argyros, G.~Godefroy and H.P.~Rosenthal in \cite[p.\,1047]{ArGoRo}. Also, several examples witnessing natural limits of our positive results are presented.
\end{abstract}
\maketitle

\section{Introduction and main results}

If $E$ is a (real or complex) Banach space, an element $x^{**}$ of its bidual may posses interesting descriptive properties if $x^{**}$ is understood as a function on the dual space endowed with the weak* topology. Since the dual unit ball $B_{E^*}$ is weak* compact, the set $\ext B_{E^*}$ of its extreme points is nonempty and its weak* closed convex hull is the whole unit ball. Hence one might expect that a behaviour of $x^{**}$ on the set $\ext B_{E^*}$ in some sense determines the behaviour of $x^{**}$ on $B_{E^*}$. The aim of our paper is to substantiate this general idea by presenting several results on transferring descriptive properties of $x^{**}\r_{\ext B_{E^*}}$ to $x\r_{B_{E^*}}$. To formulate our results precisely, we need to recall several notions.

Since the main results are mostly formulated for Banach spaces over real or complex field, we need to work with vector spaces over both real and complex numbers. So all the notions are considered, if not stated otherwise, with respect to the field of complex numbers. All topological space are considered to be Tychonoff (i.e, completely regular, see \cite[p.\,39]{EN}), in particular they are Hausdorff.

If $K$ is a compact  topological space, a \emph{positive Radon measure} on $K$ is a finite complete measure with values in $[0,\infty)$ defined at least on the $\sigma$-algebra of all Borel sets that is inner regular with respect to compact sets (see \cite[Definition~411H]{fremlin4}). A signed or complex measure $\mu$ on $X$ is a Radon measure if its total variation $|\mu|$ is Radon. We often write $\mu(f)$ instead of $\int f\, d\mu$. We denote as $\M(K)$, $\M^+(K)$ and $\M^1(K)$ the set of all Radon measures, positive Radon measures and probability Radon measures, respectively. Using the Riesz representation theorem we view $\M(K)$ as the dual space to the space $\C(K)$ of all continuous functions on $K$. Unless stated otherwise, we consider the space $\M(K)$ endowed with the weak* topology. A function $f:K\to\ce$ is \emph{universally measurable} if $f$ is $\mu$-measurable for every $\mu\in\M(K)$. If $\F$ is a family of functions, we write $\F^b$ for the set of all bounded elements of $\F$.

Let $X$ be a compact convex subset of a  locally convex space. Then any measure $\mu \in\M^1(X)$ has its unique \emph{barycenter} $x\in X$, i.e., the point $x\in X$ satisfying $\mu(f)=f(x)$ for each $f\in\fra^c(X)$ (here $\fra^c(X)$ stands for the space of all continuous affine functions on $X$). We write $\M_x(X)$ for the set of all probability measures with $x$ as the barycenter. The mapping $r:\M^1(X)\to X$ assigning to every probability measure on $X$ its barycenter is a continuous affine surjection, see \cite[Proposition~I.2.1]{alfsen} or \cite[Proposition~2.38]{book-irt}. A function $f:X\to\ce$ is called \emph{strongly affine} (or a function satisfying the \emph{barycentric formula}) if $f$ is universally measurable and $\mu(f)=f(r(\mu))$ for every $\mu\in\M^1(X)$. It is easy to deduce that any strongly affine function is bounded (see e.g. \cite[Lemma~4.5]{book-irt}).

If $E$ is Banach space, $B_{E^*}$ with the weak* topology is a compact convex set. We call an element $f\in E^{**}$ \emph{strongly affine} if its restriction to $B_{E^*}$ is a strongly affine function. We also mention that a continuous affine function $f$ on $B_{E^*}$, which satisfies $f(0)=0$ and $f(ix^*)=if(x^*)$ for $x^*\in B_{E^*}$, is in fact an element of $E$, i.e., there exists $x\in E$ with $f(x^*)=x^*(x)$ for $x^*\in B_{E^*}$.

Further we need to recall descriptive classes of functions in topological spaces. We follow the notation of \cite{spurny-ahm}.
If $X$ is a Tychonoff topological space, a \emph{zero set} in $X$ is an inverse image of a closed set in $\er$ under a continuous function $f:X\to\er$. The complement of a zero set is a \emph{cozero set}.  A countable union of closed sets is called an \emph{$F_\sigma$ set}, the complement of an $F_\sigma$ set is a \emph{$G_\delta$ set}. If $X$ is normal, it follows from Tietze's theorem that a closed set is a zero set if and only if it is also a $G_\delta$ set.
We recall that \emph{Borel sets} are members of the $\sigma$--algebra generated by the family of all open subset of $X$ and \emph{Baire sets} are members of the $\sigma$--algebra generated by  the family of all cozero sets in $X$. We write $\Bos(X)$ and $\Bas(X)$ for the algebras generated by open or cozero sets in $X$, respectively. 

A set $A\subset X$ is \emph{resolvable} (or an \emph{$H$-set}) if for any nonempty $B\subset X$ (equivalently, for any nonempty closed $B\subset X$) there exists a relatively open $U\subset B$ such that either $U\subset A$ or $U\cap A=\emptyset$. It is easy to see that the family  $\Hs(X)$ of all resolvable sets is an algebra, see e.g. \cite[\S\,12, VI]{kur}. Let $\Sigma_2(\Bas(X))$, $\Sigma_2(\Bos(X))$ and $\Sigma_2(\Hs(X))$ denote countable unions of sets from the respective algebras.

Let $\Baf_1(X)$ denote the family of all $\Sigma_2(\Bas(X))$-measurable function on $X$, i.e., the functions $f:X\to\ce$ satisfying $f^{-1}(U)\in\Sigma_2(\Bas(X))$ for all $U\subset \er$ open. Analogously we define families $\Bof_1(X)$ and $\Hf_1(X)$. 

Now we use pointwise limits to create higher hierarchies of functions. More precisely, if $\Phi$ is a family of functions on $X$, we define $\Phi_0=\Phi$ and, for each countable ordinal $\alpha$, $\Phi_\alpha$ consists of all pointwise limits of sequences from $\bigcup_{\beta<\alpha} \Phi_\beta$. Starting the procedure with $\Baf_1(X)$ and creating higher families $\Baf_\alpha(X)$ as pointwise limits of sequences contained in $\bigcup_{1\leq \beta<\alpha} \Baf_\beta(X)$, we obtain the hierarchy of \emph{Baire measurable} functions. Analogously we define, for $\alpha\in [1,\omega_1)$, families $\Bof_\alpha(X)$ and $\Hf_\alpha(X)$ of \emph{Borel measurable} functions and \emph{resolvably measurable} functions. (Theorem~5.2 in \cite{spurny-ahm} explains the term "measurability" in these definitions.)

If $X$ is a Tychonoff space and we start the inductive process with the family $\Phi_0=\Phi=\C(X)$, we obtain the families $\C_\alpha(X)$ of Baire-$\alpha$ functions on $X$, $\alpha<\omega_1$. Then the union $\bigcup_{\alpha<\omega_1}\C_\alpha(X)$ is the family of all \emph{Baire} functions. It is easy to see that $\C_1(X)=\Baf_1(X)$ (see Proposition~\ref{baf}) and thus $\C_\alpha(X)=\Baf_\alpha(X)$ for any $\alpha\in [1 ,\omega_1)$. 

Now we can state our first result concerning a preservation of descriptive properties. For separable Banach spaces and Baire functions, the results can be obtained from \cite[Corollaire~8]{sr-kvocient}.

\begin{thm}\label{main1}
Let $E$ be a (real or complex) Banach space and $f\in E^{**}$ be strongly affine. Then, 
\begin{itemize} 
\item for  $\alpha\in [1,\omega_1)$, $f\r_{\ov{\ext B_{E^*}}}\in \Hf_\alpha(\ov{B_{E^*}})$ if and only if $f\in  \Hf_\alpha( B_{E^*})$,
\item for $\alpha\in [1,\omega_1)$, $f\r_{\ov{\ext B_{E^*}}}\in \Bof_\alpha(\ov{B_{E^*}})$ if and only if $f\in  \Bof_\alpha( B_{E^*})$,
\item for  $\alpha\in [0,\omega_1)$, $f\r_{\ov{\ext B_{E^*}}}\in \C_\alpha(\ov{B_{E^*}})$ if and only if $f\in  \C_\alpha( B_{E^*})$.
\end{itemize}
\end{thm}

We remark that the assumption of strong affinity is necessary because otherwise the  transfer of properties fails spectacularly. An example witnessing this phenomenon can be constructed as follows. Consider the real Banach space $E=\C([0,1])$ and the function $f:\M([0,1])\to \er$ assigning to each $\mu\in\M([0,1])$ its continuous part evaluated at function $1$. Then $f$ is a weak* discontinuous element of $E^{**}$ contained in $\C_2(B_{\M([0,1])})$ that vanishes on $\ext B_{\M([0,1])}$. (Details can be found e.g. in \cite[Chapter~14]{PHE}, \cite[p.\,1048]{ArGoRo} or \cite[Proposition~2.63]{book-irt}.)

The next theorem in a way extend results of F. Jellett in \cite[Theorem]{Jel}.

\begin{thm}\label{main2}
Let $E$ be a (real or complex) Banach space such that $\ext B_{E^*}$ is a \lin\ set. Let $f\in E^{**}$ be a strongly affine element satisfying
$f\r_{\ext X}\in \C_\alpha(\ext B_{E^*})$ for some $\alpha\in [0,\omega_1)$. Then 
\[
f\in \begin{cases} \C_{\alpha+1}(B_{E^*}), &\alpha\in [0,\omega_0),\\
                     \C_\alpha(B_{E^*}), &\alpha\in [\omega_0,\omega_1).
      \end{cases}
\] 
\end{thm}
By assuming a stronger assumption on $\ext B_{E^*}$ we may ensure the preservation of all classes, including the finite ones.

\begin{thm}\label{main3}
Let $E$ be a (real or complex) Banach space such that $\ext B_{E^*}$ is a resolvable \lin\ set. Let $f\in E^{**}$ be a strongly affine element satisfying $f\r_{\ext X}\in \C_\alpha(\ext B_{E^*})$ for some $\alpha\in [1,\omega_1)$. Then $f\in\C_\alpha(B_{E^*})$.
\end{thm}

We remark that the shift of classes may really occur without the assumption of resolvability as it is witnessed by Example~\ref{noFsigma}.
One may also ask whether results analogous to the ones of Theorems~\ref{main2} and~\ref{main3} remains true for functions from classes $\Bof_\alpha$ and $\Hf_\alpha$. Examples~\ref{ex1} and~\ref{ex2} show that this is not the case.

Further we observe that, for a separable space $E$, the topological condition imposed on $\ext B_{E^*}$ in Theorem~\ref{main3} is equivalent with the requirement that $\ext B_{E^*}$ is a set of type $F_\sigma$. This can be seen from the following two facts:  a subset of a compact metrizable space is a resolvable set if and only if it is both of type $F_\sigma$  and $G_\delta$ (use \cite[\S\,26, X]{kur} and the Baire category theorem); the set of extreme points in a metrizable compact convex set is of type $G_\delta$ (see \cite[Corollary I.4.4]{alfsen} or \cite[Proposition~3.43]{book-irt}).

We also point out that the topological assumption in Theorem~\ref{main3} is satisfied provided $\ext B_{E^*}$ is an $F_\sigma$ set. To see this, we first  notice that $\ext B_{E^*}$ is then a Lindel\"of space. Second, we need to check that $\ext B_{E^*}$ is a resolvable set in $B_{E^*}$. To this end, assume that $F\subset B_{E^*}$ is a nonempty closed set such that both $F\cap \ext B_{E^*}$ and $F\setminus \ext B_{E^*}$ are dense in $F$. By \cite[Th\'eor\`eme 2]{talagr2}, we can write
\[
\ext B_{E^*}=\bigcap_{n=1}^\infty (H_n\cup V_n),
\]
where $H_n\subset B_{E^*}$ is closed and $V_n\subset B_{E^*}$ is open, $n\in\en$.
Thus both $F\setminus \ext B_{E^*}$ and $F\cap \ext B_{E^*}$ are comeager disjoint sets in $F$, a contradiction with the Baire category theorem.
Hence $\ext B_{E^*}$ is a resolvable set.
 
For the particular class of Banach spaces, namely $L_1$-preduals, one can obtain an information on an affine class of a function from its descriptive class (we recall that a Banach space is an \emph{$L_1$-predual} if $E^*$ is isometric to some space $L_1(\mu)$; see \cite[p.\,59]{JO-LI}, \cite[Chapter~7]{lacey} or \cite[Section~II.5]{Werbook}). Affine classes $\fra_\alpha(X)$, $\alpha<\omega_1$, of functions on a compact convex set $X$  are created inductively from $\fra_0(X)=\fra^c(X)$ (see \cite{capon} or \cite[Definition~5.37]{book-irt}). We also remark that a pointwise convergent sequence of affine functions on $X$ is uniformly bounded which easily follows from the uniform boundedness principle (see e.g. \cite[Lemma~5.36]{book-irt}), and thus any function in $\bigcup_{\alpha<\omega_1} \fra_\alpha(X)$ is strongly affine. If $X=B_{E^*}$ is the dual unit ball of a Banach space $E$, the affine classes are termed \emph{intrinsic Baire classes} of $E$ in \cite[p.\,1047]{ArGoRo} whereas strongly affine Baire functions on $X$ creates hierarchy of  \emph{Baire classes} of $E$.
Theorem~\ref{main4} relates these classes for real $L_1$-preduals.

We recall that, given a compact convex set $X$ in a real locally convex space, the real Banach space $\fra^c(X)$ is an $L_1$-predual if and only if $X$ is a \emph{simplex}, i.e., if for any $x\in X$ there exists a unique maximal measure $\delta_x\in \M^1(X)$ representing $x$ (see \cite[Theorem~3.2 and Proposition~3.23]{FoLiPh}).

(A measure $\mu\in\M^+(X)$ is \emph{maximal} if $\mu$ is maximal with respect to the Choquet ordering, i.e., $\mu$ fulfils the following condition: if a measure $\nu\in\M^+(X)$ satisfies $\mu(k)\leq \nu(k)$ for any convex continuous function $k$ on $X$, then $\mu=\nu$. We refer the reader to \cite[Chapter~I\,,\S\,3]{alfsen} or \cite[Section~3.6]{book-irt} for information on maximal measures.)

\begin{thm}\label{main4}
Let $E$ be a real $L_1$-predual and $f\in E^{**}$ be a strongly affine function such that $f\in\C_\alpha(B_{E^*})$ for some $\alpha\in [2,\omega_1)$. Then
\[
f\in\begin{cases} \fra_{\alpha+1}(B_{E^*}),& \alpha\in [2,\omega_0),\\
                  \fra_\alpha(B_{E^*}),&\alpha\in [\omega_0,\omega_1).
      \end{cases}                          
\]
If, moreover, $\ext B_{E^*}$ is a \lin\ resolvable set, then $f\in \fra_{\alpha}(B_{E^*})$.
\end{thm}

Let us point out that, for any Banach space $E$ and a strongly affine function $f\in E^{**}$ satisfying $f\in \C_1(B_{E^*})$, we have $f\in\fra_1(B_{E^*})$. This follows from \cite[Th\'eor\`eme~80]{Rog} (see also \cite[Theorem~II.1.2]{ArGoRo} or \cite[Theorem~4.24]{book-irt}). For higher Baire classes, there is a big gap between  affine and Baire classes which is an assertion substantiated by M. Talagrand's example \cite[Theorem]{talagr6} where he constructed a separable Banach space $E$ and a strongly affine function $f\in E^{**}$ that is in $\C_2(B_{E^*})$ and not contained in $\bigcup_{\alpha<\omega_1} \fra_\alpha(B_{E^*})$.
Further, \cite[Theorem~1.1]{spurny-trans} shows that the shift of classes in Theorem~\ref{main4} for finite ordinals may occur even for separable $L_1$-preduals.

The strategy of the proofs of our main results is to reduce firstly the problem to the case of real Banach spaces and then to consider the dual unit ball with the weak* topology as a compact convex subset of a real locally convex space. Elements of the bidual are then bounded affine functions on the dual unit ball. The key results of Sections~\ref{s:cl-boundary}--\ref{s:transf-resol} are thus formulated for this setting. The proof of Theorem~\ref{main4} is moreover based upon a result of W.~Lusky stating that any real $L_1$-predual is complemented in a simplex space (i.e., a space of type $\fra^c(X)$ for a simplex $X$) and thus our above mentioned technique can be used only for real $L_1$-preduals.  Since it is not clear whether Lusky's result remains true for complex $L_1$-preduals, \emph{the validity of Theorem~\ref{main4} for complex spaces remains open}.

The content of our paper is the following. The second section provides a more detailed information on descriptive classes of sets and functions. Then we prepare a proof of Theorem~\ref{main1} in Section~\ref{s:cl-boundary}. Results necessary for dealing with Lindel\"of sets of extreme points are collected in Section~\ref{s:lind}. They are used in Sections~\ref{s:transf-lind} and~\ref{s:transf-resol}, which prepares ground for the proof of Theorems~\ref{main2} and~\ref{main3}. All Sections~\ref{s:cl-boundary}--\ref{s:transf-resol} deal within the context of real spaces. Section~\ref{s:main} proves by means of prepared results the theorems stated in the introduction. The last Section~\ref{examples} constructs spaces witnessing some natural bounds of our positive results.

When citing references, we try to include several sources to help the reader with finding relevant results. 

\section{Descriptive classes of sets and functions}
\label{s:des}

We recall that, for a Tychonoff space $X$, $\Bas(X)$, $\Bos(X)$ and $\Hs(X)$ denote the algebras generated by cozero sets, open sets and resolvable sets in $X$, respectively. These algebras serve as a starting point of an inductive definition of descriptive classes of sets as was indicated in introduction. More precisely, if $\F$ is any of the families above, $\Sigma_2(\F)$ consists of all countable unions of sets from $\F$ and $\Pi_2(\F)$ of all countable intersections of sets from $\F$. Proceeding inductively, for any $\alpha\in (2,\omega_1)$ we let $\Sigma_\alpha(\F)$ to be made of all countable unions of sets from $\bigcup_{1\leq \beta<\alpha}\Pi_\beta(\F)$ and $\Pi_\alpha(\F)$ is made of all countable intersections of sets from $\bigcup_{1\leq \beta<\alpha}\Sigma_\beta(\F)$. The family  $\Pi_\alpha(\F)\cap \Sigma_\alpha(\F)$ is denoted as $\Delta_\alpha(\F)$. The union of all created additive (or multiplicative) classes is then the $\sigma$-algebra generated by $\F$. 

(These classes and their analogues were studied by several authors, see e.g. \cite{hansell}, \cite{raja}, \cite{ho-pe} or \cite{ho-spa}. We describe in \cite[Remark~3.5]{spurny-ahm} their relations to our descriptive classes. We refer the reader to \cite{ho-spa} for a recent survey on descriptive set theory in nonseparable and nonmetrizable spaces.)

In case $X$ is metrizable, all the resulting classes coincide (see \cite[Proposition~3.4]{spurny-ahm}). These classes characterize in terms of measurability the classes $\Baf_\alpha(X)$, $\Bof_\alpha(X)$ and $\Hf_\alpha(X)$ defined in the introduction. (We recall that a mapping $f:X\to \ce$ is called \emph{$\F$-measurable} if $f^{-1}(U)\in \F$ for every $U\subset \ce$ open.)  Precisely, it is proved in \cite[Theorem~5.2]{spurny-ahm} that \emph{given a function $f:X\to\ce$ on a Tychonoff space $X$ and $\alpha\in [1,\omega_1)$, we have
\begin{itemize}
\item  $f\in\Baf_\alpha(X)$ if and only if $f$ is $\Sigma_{\alpha+1}(\Bas(X))$--measurable.
\item  $f\in\Bof_\alpha(X)$ if and only if $f$ is $\Sigma_{\alpha+1}(\Bos(X))$--measurable.
\item  $f\in\Hf_\alpha(X)$ if and only if $f$ is $\Sigma_{\alpha+1}(\Hs(X))$--measurable.
\end{itemize}
}
It follows easily from this characterization that all the classes $\Baf_\alpha(X)$, $\Bof_\alpha(X)$ and $\Hf_\alpha(X)$ are stable with respect to algebraic operations and uniform convergence (see \cite[Theorem~5.10]{book-irt}).
Also, a function $f$ is measurable with respect to the $\sigma$-algebra generated by $\Hs$ if and only if $f$ belongs to some class $\Hf_\alpha$. Analogous assertions hold true for the algebras $\Bos$ and $\Bas$. Thus $\bigcup_{\alpha<\omega_1} \C_\alpha(X)=\bigcup_{\alpha<\omega_1}\Baf_\alpha(X)$ is the family of all functions measurable with respect to the $\sigma$-algebra of Baire sets.

The following characterization  of functions from $\Hf_1$ follows from the definition and results of G.~Koumoullis in \cite[Theorem~2.3]{koum}.

\begin{prop}\label{propHf}
For a function $f:K\to\ce$ on a compact space $K$, the following assertions are equivalent:
\begin{enumerate}
\item [(i)] $f\in\Hf_1(K)$,
\item [(ii)] $f\r_F$ has a point of continuity  for every closed $F\subset K$ (i.e., $f$ has the \emph{point of continuity property}),  
\item [(iii)] for each $\ep>0$ and nonempty $F\subset X$ there exists a relatively open nonempty set $U\subset F$ such that $\diam f(U)<\ep$ ($f$ is \emph{fragmented}).
\end{enumerate}
\end{prop}

Next we need to recall a characterization of resolvable sets that asserts that \emph{a subset $H$ of a topological space $X$ is resolvable if and only if there exist an ordinal $\kappa$ and an increasing sequence of open sets $\emptyset = U_0 \subset U_1 \subset U_2 \subset \cdots \subset U_\gamma \subset \cdots \subset U_\kappa = X$ and $I \subset [0,\kappa)$ such that, for a limit ordinal $\gamma \in [0,\kappa]$, we have $\bigcup \{ U_{\lambda}: \lambda < \gamma \} = U_\gamma$ and  $H = \bigcup \{ U_{\gamma+1} \setminus U_{\gamma} : \gamma \in I \}$} (see \cite[Section~2]{HoSp} and references therein). 
We call such a transfinite sequence of open sets \emph{regular} and such a description of a resolvable set a \emph{regular representation} (this notion of regular representation is slightly more useful for us than the one used in  \cite[Section~2]{HoSp}).

A family $\mathcal U$ of subsets of a topological space $X$ is \emph{scattered}
if it is disjoint and for each nonempty $\mathcal V\subset\mathcal U$ there is
some $V\in \mathcal V$ relatively open in $\bigcup\mathcal V$. If $(U_\gamma)_{\gamma\leq \kappa}$ is a regular sequence, then $\{U_{\gamma+1}\setminus U_{\gamma}: \gamma<\kappa\}$ is a scattered partition of $X$.

It is not difficult to deduce that a scattered union of resolvable sets is again a resolvable set. 
(Indeed, let $\{H_i: i\in I\}$ be a scattered family of resolvable sets. By \cite[Fact 4]{ho-pe}, each $H_i$ is a union of a scattered family $\H_{i}$ of sets in $\Bos(X)$. By \cite[Lemma~2.2(c)]{hansell}, the family $\bigcup_{i\in I} \H_{i}$ is scattered, and thus again by \cite[Fact~4]{ho-pe}, the set $\bigcup_{i\in I} H_i$ is resolvable.)

We will also need a fact that any resolvable subset of a compact space is universally measurable (see \cite[Lemma~4.4]{koum}).
 
The following fact will be used in the proof of Theorem~\ref{porce}.

\begin{prop}\label{propH}
Let $\alpha\in [2,\omega_1)$ and $(U_\gamma)_{\gamma\leq \kappa}$ be a regular sequence in a Tychonoff space $X$. Let $A\subset X$ be such that $A\cap (U_{\gamma+1}\setminus U_\gamma)\in \Sigma_\alpha(\Hs(U_{\gamma+1}\setminus U_\gamma))$ for each $\gamma<\kappa$ (or $A\cap(U_{\gamma+1}\setminus U_\gamma)\in \Pi_\alpha(\Hs(U_{\gamma+1}\setminus U_\gamma))$, $\gamma<\kappa$). Then $A\in\Sigma_\alpha(\Hs(X))$ (or $A\in \Pi_\alpha(\Hs(X))$).
\end{prop}

\begin{proof}
If $\alpha=2$, the assertion for the additive class follows from the fact mentioned above that a scattered union of resolvable sets is again a resolvable sets. By taking complements we obtain the assertion for $\Pi_2(\Hs)$.
A straightforward transfinite induction then concludes the proof.
\end{proof}

For the sake of ompleteness, we include a proof of an easy observation mentioned in the introduction.

\begin{prop}
\label{baf}
If $X$ is a Tychonoff space, $\C_1(X)=\Baf_1(X)$.
\end{prop}

\begin{proof}
If $f\in\C_1(X)$, a straightforward reasoning gives $f\in\Baf_1(X)$. On the other hand, if $f\in\Baf_1(X)$, it is enough to assume that $f$ is real-valued. If $f$ is moreover bounded, a  standard procedure (see e.g. \cite[Lemma~5.7]{book-irt}) provides a uniform approximation by a sequence of simple functions, i.e., functions of the form $\sum_{i=1}^n c_i\chi_{A_i}$, where $c_1,\dots,c_n\in\er$ and $\{A_1,\dots,A_n\}$ is a disjoint cover of $X$ such that each $A_i$ is a countable unions of zero sets. A moment's reflection reveals that any such function is in $\C_1(X)$. Hence $f\in\C_1(X)$ as well.

If $f$ is unbounded, we take a homeomorphism $\varphi:\er\to (0,1)$ and apply the procedure above to $\varphi\circ f\in\Baf_1(X)$ to infer $\varphi\circ f\in\C_1(X)$.
We can then arrange an approximating sequence $(f_n)$ of continuous functions on $X$ in such a way that $0<f_n<1$, $n\in\en$. Then $\varphi^{-1}\circ f_n\to f$, and $f\in\C_1(X)$. 
\end{proof}

\section{Transfer of descriptive properties from $\ov{\ext X}$ to $X$}
\label{s:cl-boundary}

Throughout this section we work with real spaces. The main result is Theorem~\ref{transf-cl} on transferring descriptive properties of strongly affine functions from the closure of the set of extreme points.

\begin{lemma}\label{transferHset}
Let $K$ be a compact space and $H$ a universally measurable subset of $K$. Let $\wt{f}\colon \M^1(K)\to\er$ be defined as $\wt{f}(\mu) = \mu(H)$, $\mu\in\M^1(K)$. Then
\begin{itemize}
	\item $\wt{f} \in \Hf_1(\M^1(K))$ if $H\in \Hs(K)$,
	\item $\wt{f} \in \Bof_1(\M^1(K))$ if $H\in \Bos(K)$.
\end{itemize}
\end{lemma}
\begin{proof}
We first assume that $H$ is a resolvable set. We select a regular sequence $(U_\gamma)_{\gamma\leq \kappa}$ which provides a regular representation of $H$ as mentioned in Section~\ref{s:des}. We prove  by transfinite induction that, 
\emph{for every $\gamma\leq \kappa$, the function $\mu\mapsto\mu(H\cap U_\gamma)$ is in $\Hf_1(\M^1(K))$.}  

The statement holds trivially for $\gamma = 0$.

We suppose now that $\gamma\leq \kappa$ is of the form $\gamma = \delta + 1$ and the claim is valid for $\delta$. Then, for every $\mu \in \M^1(K)$, we have
\[
\mu(H\cap U_\gamma) = \mu ( H \cap U_{\delta}) + \mu( H \cap (U_{\delta + 1} \setminus U_{\delta})).
\]
The second summand is either equal to $0$ or $\mu ( U_{\delta+1}) - \mu ( U_{\delta})$. Since the function $\mu\mapsto\mu(U)$ is lower semicontinuous on $\M^1(K)$ for every open set $U\subset K$, it follows e.g. from \cite[Theorem~2.3]{koum} that the function $\mu\mapsto \mu(U_{\delta+1})-\mu(U_\delta)$ is in $\Hf_1(\M^1(K))$.

The function $\mu \to \mu(H\cap U_\delta)$ is in $\Hf_1(\M^1(K))$ due to the induction hypothesis.
Thus $\mu\mapsto \mu(H)$, as a sum of two functions in $\Hf_1(\M^1(K))$, is in $\Hf_1(\M^1(K))$ as well. 

Assume now that $\gamma\leq \kappa$ is a limit ordinal and the statement holds for each ordinal smaller than $\gamma$. 
Let $\wt{f}(\mu)=\mu(H\cap U_\gamma)$, $\mu\in\M^1(K)$.
Assuming $\wt{f}$ is not in $\Hf_1(\M^1(K))$, Proposition~\ref{propHf} provides a nonempty set $M\subset \M^1(K)$ and $\ep>0$ such that $\diam \wt{f}(M\cap V)> \ep$ for each open set $V\subset \M^1(K)$ intersecting $M$. 
Let 
\[
s=\sup\{\mu(U_\gamma)\colon \mu\in M\}
\]
and let $\mu_0\in M$ be chosen such that $\mu_0(U_\gamma)>s-\frac{\ep}{4}$. By the regularity of $\mu_0$, there exists $\delta<\gamma$ with $\mu_0(U_\delta)>s-\frac{\ep}{4}$.
Then the set 
\[
V=\{\mu \in \M^1(K): \mu(U_{\delta}) > s-\frac{\ep}4\}
\]
is an open neighborhood of $\mu_0$. 

Let $\wt{h}:\M^1(K) \to\er$ be defined as $\wt{h}(\mu)=\mu(H\cap U_{\delta})$. 
Then we have
\[
|\wt{h}(\mu) - \wt{f}(\mu)| = | \mu(H\cap U_{\delta}) - \mu(H\cap U_\gamma)| \leq |\mu(U_{\gamma}\setminus U_\delta)|\leq
s-(s-\frac{\ep}{4})=\frac{\ep}{4},\quad \mu\in M\cap V,
\]
and, by the induction hypothesis,  $\wt{h}$ is in $\Hf_1(\M^1(K))$ which means that $\wt{h}$ is fragmented.

Thus there exists an open set $W\subset \M^1(K)$ intersecting $M\cap V$ such that $\diam \wt{h}(M\cap V \cap W) < \frac{\ep}{4}$. 
By the assumption, there exist $\mu_1,\mu_2 \in M\cap V\cap W$ satisfying $|\wt{f}(\mu_1)-\wt{f}(\mu_2)| \geq \ep$. On the other hand we have
\[
|\wt{f}(\mu_1) - \wt{f}(\mu_2)| \leq | \wt{f}(\mu_1) - \wt{h}(\mu_1)| + |\wt{h}(\mu_1) -\wt{h}(\mu_2)| + |\wt{h}(\mu_2) - \wt{f}(\mu_2)| \leq \frac34 \ep,
\]
which is a contradiction. Thus $\wt{f}$ is fragmented. This proves the claim as well as the proof of the first assertion.

Assume now that $H\in \Bos(K)$. Then $H$ can be written as a finite disjoint union of differences of closed sets (see e.g. \cite[Lemma~5.12]{book-irt}), i.e., $H=\bigcup_{i=1}^n E_i\setminus F_i$, where $F_i\subset E_i$ are closed and the family $\{E_1\setminus F_1,\dots, E_n\setminus F_n\}$ is  disjoint. Then the function $\mu\mapsto \mu(E_i\setminus F_i)$, as a difference of a couple of upper semicontinuous functions on $\M^1(K)$, is in $\Bof_1(\M^1(K))$ for each pair $E_i, F_i$.

Hence $\mu\mapsto\mu(H)$, $\mu\in\M^1(K)$, is a finite union of functions in $\Bof_1(\M^1(K))$, and thus contained in $\Bof_1(\M^1(K))$.
\end{proof}

\begin{lemma}\label{tran-f}
Let $f:K\to \er$ be a bounded universally measurable function and let $\wt{f}\colon \M^1(K)\to\er$ be defined as $\wt{f}(\mu)=\mu(f)$, $\mu\in\M^1(K)$. Then 
\begin{itemize}
	\item $\wt{f} \in \Hf_1(\M^1(K))$ if $f \in \Hf_1(K)$,
	\item $\wt{f} \in \Bof_1(\M^1(K))$ if $f \in \Bof_1(K)$.
\end{itemize}
\end{lemma}

\begin{proof}
We begin with the proof for $f\in\Hf_1(K)$. First, if $f = \chi_{A}$ is the characteristic function of a set $A \in \Delta_2(\Hs(K))$, we write $A=\bigcup_n A_n$, where $A_1\subset A_2\subset \cdots$ are sets in $\Hs(K)$.  If $c\in \er$ is given, we have from Lemma \ref{transferHset} that
\[
\{\mu\in \M^1(K) : \wt{f}(\mu) > c \} = \bigcup_{n=1}^{\infty} \{ \mu\in\M^1(K): \mu(A_n)>c \} \in \Sigma_2(\Hs(K)).
\]
On the other hand, $K \setminus A \in \Sigma_2(\Hs(K))$ and hence it follows from the previous reasoning that 
\[
\{\mu\in \M^1(K) : \wt{f}(\mu) < c \} = \{ \mu\in\M^1(K): \mu(K\setminus A) > 1-c \} \in \Sigma_2(\Hs(K)).
\]
We conclude that $\wt{f}$ is $\Sigma_2(\Hs(\M^1(K)))$-measurable and hence $\wt{f}\in \Hf_1(\M^1(K))$. 

If $f\in\Hf_1(K)$ is bounded, it can be uniformly approximated by simple functions in $\Hf_1(K)$, i.e., functions of the form $\sum_{i=1}^n c_i \chi_{A_i}$, where $A_1,\dots, A_n\in\Delta_2(\Hs(K))$ are pairwise disjoint and $c_1,\dots, c_n\in \er$ (this standard procedure can be found e.g. in \cite[Lemma~5.7]{book-irt}). Hence $\wt{f}$ can be uniformly approximated by functions in $\Hf_1(\M^1(K))$, and thus $\wt{f}\in\Hf_1(\M^1(K))$.

The proof for $f\in \Bof_1(K)$ would proceed in a similar fashion.
\end{proof}

\begin{lemma}\label{transM}
Let $K$ be a compact space and $f:K\to \er$ be a bounded universally measurable function. Let $\wt{f}\colon \M^1(K)\to\er$ be defined as $\wt{f}(\mu)=\mu(f)$, $\mu\in\M^1(K)$. Then,
\begin{itemize} 
\item [(a)] for  $\alpha\in [1,\omega_1)$, $f\in \Hf_\alpha(K)$ if and only if $\wt{f}\in  \Hf_\alpha(\M^1(K))$,
\item [(b)] for  $\alpha\in [1,\omega_1)$, $f\in \Bof_\alpha(K)$ if and only if $\wt{f}\in  \Bof_\alpha(\M^1(K))$,
\item [(c)] for  $\alpha\in [0,\omega_1)$, $f\in \C_\alpha(K)$ if and only if $\wt{f}\in \C_\alpha(\M^1(K))$.
\end{itemize}
\end{lemma}

\begin{proof}
The "if" parts of the proof easily follows from the fact $ f = \wt{f} \circ \phi$ where $\phi:K\to \M^1(K)$ sending a point $x\in K$ to the Dirac measure $\ep_x$ at $x$ is a homeomorphic embedding.

The proof of "only if" part will be given by transfinite induction. If $\alpha=1$ in (a) and (b), the assertion follows from Lemma~\ref{tran-f}, the case $\alpha=0$ in (c) is obvious.

The assertions for higher ordinals $\alpha$ now follows by a straightforward induction.
\end{proof}

As we mentioned in the introduction, the following theorem is a generalization of \cite[Corollaire~8]{sr-kvocient}.

\begin{thm}\label{transf-cl}
Let $X$ be a compact convex set and $f:X\to\er$ be a strongly affine function. Then, 
\begin{itemize} 
\item for  $\alpha\in [1,\omega_1)$, $f\r_{\ov{\ext X}}\in \Hf_\alpha(\ov{\ext X})$ if and only if $f\in  \Hf_\alpha(X)$,
\item for  $\alpha\in [1,\omega_1)$, $f\r_{\ov{\ext X}}\in \Bof_\alpha(\ov{\ext X})$ if and only if $f\in  \Bof_\alpha(X)$,
\item for  $\alpha\in [0,\omega_1)$, $f\r_{\ov{\ext X}}\in \C_\alpha(\ov{\ext X})$ if and only if $f\in  \C_\alpha(X)$.
\end{itemize}
\end{thm}

\begin{proof}
It is easy to realize that all the families $\Hf_\alpha$, $\Bof_\alpha$ and $\C_\alpha$ are preserved by making restrictions to subspaces of $X$. This observation gives the "if" parts of the proof.

For the proof of the "only if" parts, let $f:X\to\er$ be a strongly affine function with $f\r_{\ov{\ext X}}\in \F(\ov{\ext X})$ where $\F$ is any of the classes $\Hf_\alpha$, $\Bof_\alpha$ or $\C_\alpha$. 
Then the function $\wt{g}: \M^1(\ov{\ext X}) \to \er$ defined as 
\[
\wt{g}(\mu) = \mu(f),\quad \mu\in\M^1(\ov{\ext X}),
\]
is in $\F(\M^1(\ov{\ext X}))$ by Lemma~\ref{transM}.

The mapping $r: \M^1(\ov{\ext X}) \to X$, which assigns $\mu\in \M^1(\ov{\ext X})$ its barycenter $r(\mu)\in X$, is a continuous surjection of a compact space $\M^1(\ov{\ext X})$ onto $X$ (see \cite[Proposition~I.4.6 and Theorem~I.4.8]{alfsen} or \cite[Theorem 3.65 and Proposition~3.64]{book-irt}).

From the strong affinity of $f$ we have $\wt{g} = f \circ r$.  Now we use the fact that $\wt{g}\in \F(\M^1(\ov{\ext X}))$ if and only if $f\in \F(X)$. This fact can be found in \cite[Theorem~5.9.13]{ROJA} and \cite[Theorem~5.26]{book-irt} for classes $\C_\alpha$, and in \cite[Theorems~4 and~10]{HoSp} for classes $\Bof_\alpha$ and $\Hf_\alpha$ (see also \cite[Theorem~5.26]{book-irt}). 
Thus the function $f$ is in $\F(X)$.
\end{proof}


\section{Auxiliary result on compact convex sets with $\ext X$ being \lin}
\label{s:lind}

Throughout this section we work with spaces over the field of real numbers. We aim for the proof of Proposition~\ref{lind0.03} which is a fact used both in Section~\ref{s:transf-lind} and~\ref{s:transf-resol}. We recall that a topological space $X$ is \emph{$K$-analytic} if it is an image of a Polish space under an upper semicontinuous compact-valued map (see \cite[Section~2.1]{ROJA}).

\begin{lemma}\label{lind0.01}
Let $\varphi\colon X\to Y$ be a continuous surjection of a $K$-analytic space $X$ onto a $K$-analytic space $Y$ and let $g:Y\to\er$. Then $g$ is a Baire function on $Y$ if and only if $g\circ\varphi$ is a Baire function on $X$.
\end{lemma}

\begin{proof}
If $g$ is a Baire function $Y$, then $g\circ\varphi$ is clearly a Baire function on $X$. Conversely, if $f=g\circ\varphi$ is a Baire function on $X$ and $U\subset \er$ is an open set, then both $f^{-1}(U)$ and $f^{-1}(\er\setminus U)$ are Baire sets in $X$. Then they are $K$-analytic sets in $X$ (see \cite[Section~2]{ROJA}), and thus 
\[
g^{-1}(U)=\varphi(f^{-1}(U)),\quad g^{-1}(\er\setminus U)=\varphi(f^{-1}(\er\setminus U))
\]
are $K$-analytic as well. It follows from the proof of the standard separation theorem (see \cite[Theorem~3.3.1]{ROJA}) that they are Baire sets.
Hence $g$ is measurable with respect to the $\sigma$-algebra of Baire sets, and thus it is a Baire function.
\end{proof}

\begin{lemma}\label{lind0.02}
Let $f\colon X\to \er$ be a strongly affine function on a compact convex set $X$ for which there exists a Baire set $B\supset \ext X$ such that $f\r_B$ is a Baire function. Then $f$ is a Baire function on $X$.
\end{lemma}

\begin{proof}
Let $B\supset \ext X$ and $f:X\to\er$ be as in the hypothesis. Let 
\[
\wt{B}=\{\mu\in\M^1(X)\colon \mu(B)=1\}.
\]
Since the characteristic function of $B$ is a Baire function, the function $\wt{c}(\mu)=\mu(B)$, $\mu\in\M^1(X)$, is a Baire function on $\M^1(X)$ as well, and thus $\wt{B}=\{\mu\in\M^1(X)\colon \wt{c}(\mu)=1\}$ is a Baire set in $\M^1(X)$. Hence $\wt{B}$ is a $K$-analytic space and it follows from Lemma~\ref{transM}(c) that the function $\wt{f}:\wt{B}\to \er$ defined as 
\[
\wt{f}(\mu)=\mu(f),\quad \mu\in\wt{B},
\]
is a Baire function on $\wt{B}$. 

Then $r:\wt{B}\to X$ is a continuous surjective mapping satisfying $\wt{f}=f\circ r$ (see \cite[Corollary~I.4.12 and the subsequent remark]{alfsen} or \cite[Theorem~3.79]{book-irt}). By Lemma~\ref{lind0.01}, $f$ is a Baire function.
\end{proof}

\begin{lemma}\label{lind0.020}
Let $X$ be a compact convex set with $\ext X$ \lin, $\mu\in\M^1(X)$ be maximal and $B\supset \ext X$ be $\mu$-measurable. Then $\mu(B)=1$.
\end{lemma}

\begin{proof}
Given $B\supset \ext X$ and maximal measure $\mu\in\M^1(X)$, by the regularity of $\mu$ it is enough to show that $\mu(K)=0$ for every $K\subset X\setminus B$ compact. Given such a set $K$, for every $x\in \ext X$ we select a closed neighborhood $U_x$ of $x$ disjoint from $K$. By the \lin\ property we choose a countable set $\{x_n\colon n\in\en\}\subset \ext X$ with $\ext X\subset \bigcup U_{x_n}$. By Corollary~I.4.12 and the subsequent remark in \cite{alfsen} (see also \cite[Theorem~3.79]{book-irt}), $\mu(\bigcup U_{x_n})=1$. Hence $\mu(K)=0$, which concludes the proof.
\end{proof}

\begin{lemma}\label{lin-appr}
Let $X$ be a compact convex set with $\ext X$ \lin\ and $f\in\C^b(\ext X)$. Then there exist a decreasing sequence $(u_n)$ of continuous concave functions on $X$ and an increasing sequence $(l_n)$ of continuous convex functions on $X$ such that 
\[
\inf f(\ext X)\leq \inf l_1(X),\quad \sup u_1(X)\leq \sup f(\ext X),
\]
and 
\[
u_n\searrow f,\ l_n\nearrow f\text{ on }\ext X.
\]
\end{lemma}

\begin{proof}
Without loss of generality we may assume that 
\[
0\leq i=\inf f(X)\leq \sup f(X)=s\leq 1\quad\text{on }\ext X.
\]
We construct a decreasing sequence $(u_n)$ of continuous concave functions on $X$ with values in $[0,1]$ such that $u_n\searrow f$ on $\ext X$.
To achieve this, we define $h\colon \ov{\ext X}\to [0,1]$ as
\[
h(x)=\begin{cases} f(x),& x\in\ext X,\\
                        \limsup_{y\to x, y\in\ext X} f(y),& x\in\ov{\ext X}\setminus \ext X.
           \end{cases}             
\]
Then $h$ is upper semicontinuous on $\ov{\ext X}$ and the function
\[
h^*=\inf\{a\in\fra^c(X)\colon a\geq f\text{ on }\ext X\}
\]
satisfies $h=h^*=f$ on $\ext X$ by \cite[Proposition~I.4.1]{alfsen} (see also \cite[Theorem~3.24]{book-irt}).
Hence 
\[
f=\inf\{a\in\fra^c(X)\colon a\geq f\text{ on }\ext X\}\quad\text{ on }\ext X.
\]
Since $\ext X$ is a \lin\ space, there exists a countable family $\H=\{h_n\colon n\in\en\}$ of functions in $\fra^c(X)$ majorizing $f$ on $\ext X$ such that $f=\inf \H$ on $\ext X$ (see \cite[Lemma]{Jel} or \cite[Lemma~A.54]{book-irt}).
Then we obtain the desired sequence by setting 
\[
u_1=s\wedge h_1,\ u_2=s\wedge h_1\wedge\cdots\wedge h_n,\ \dots, \quad n\in\en.
\]
Analogously we obtain an increasing sequence $(l_n)$ of convex continuous functions converging to $f$ on $\ext X$.
\end{proof}

\begin{lemma}\label{lind0.021}
Let $X$ be a compact convex set with $\ext X$ \lin\ and let $f\in\C_\alpha(\ext X)$ have values in $[0,1]$. Then there exist a Baire set $B\supset \ext X$ and a function $g\in\C_\alpha(B)$ such that 
\begin{itemize}
\item $g=f$ on $\ext X$,
\item $0\leq g\leq 1$ on $B$, and 
\item $g(r(\mu))=\mu(g)$ for any $\mu\in\M^1(X)$ satisfying $\mu(B)=1$ and $r(\mu)\in B$.
\end{itemize}
\end{lemma}

\begin{proof}
We proceed by transfinite induction on the class of a function $f$.

Assume first that $f$ is continuous on $\ext X$.
Using Lemma~\ref{lin-appr} we find relevant sequences $(u_n)$ and $(l_n)$, and define $u=\inf_{n\in\en} u_n$, $l=\sup_{n\in\en} l_n$. Then  we observe that $l\leq u$ by the minimum principle (see \cite[Theorem~I.5.3]{alfsen} or \cite[Theorem~3.16]{book-irt}, both functions are Baire, $u$ is upper semicontinuous concave and $l$ is lower semicontinuous convex. Let
\[
B=\{x\in X\colon u(x)=l(x)\}\quad\text{and}\qquad g(x)=u(x),\quad x\in B.
\]
Then $B$ is a Baire set containing $\ext X$ and, for $x\in B$ and $\mu\in\M_x(X)$ with $\mu(B)=1$, we have
\[
g(x)=u(x)\geq \mu(u)=\mu(l)\geq l(x)=g(x).
\]
Since $g$ is continuous on $B$, the proof is finished for the case $\alpha=0$.

Assume now that the claim holds true for all $\beta$ smaller then some countable ordinal $\alpha$. Given $f\in\C_\alpha(\ext X)$ with values in $[0,1]$, let $(f_n)$ be a sequence of functions with $f_n\in \C_{\alpha_n}(\ext X)$ for some $\alpha_n<\alpha$, $n\in\en$, such that $f_n\to f$. Without loss of generality we may assume that all functions $f_n$ have values in $[0,1]$. For each $n\in\en$, we use the induction hypothesis and find a Baire set $B_n\supset \ext X$ along with a function $g_n\in\C_{\alpha_n}(B_n)$ with values in $[0,1]$ that coincides with $f_n$ on $\ext X$ and satisfies $g_n(r(\mu))=\mu(g_n)$ for any $\mu\in\M^1(X)$ satisfying $\mu(B_n)=1$ and $r(\mu)\in B_n$.

We set
\[
B=\{x\in \bigcap_{n=1}^\infty B_n\colon (g_n(x))\text{ converges}\}\quad\text{and}\quad g(x)=\lim_{n\to\infty} g_n(x),\ x\in B.
\]
Then $B$ is  Baire set containing $\ext X$, $g\in\C_\alpha(B)$ with values in $[0,1]$, 
\[
g_n(x)=f_n(x)\to f(x)\quad\text{for every }x\in\ext X,
\]
and, for $x\in B$ and $\mu\in\M_x(X)$ with $\mu(B)=1$, 
\[
g(x)=\lim_{n\to \infty} g_n(x)=\lim_{n\to \infty} \mu(g_n)=\mu(g).
\]
This finishes the proof.
\end{proof}

\begin{lemma}\label{g=f}
Let $X$ be a compact convex set with $\ext X$ \lin\ and let $f\colon X\to\er$ be a strongly affine function such that $f\r_{\ext X}\in\C_\alpha(\ext X)$. Then there exists a Baire set $B\supset \ext X$ such that $f\in\C_\alpha(B)$.
\end{lemma}

\begin{proof}
Given a function $f$ as in the hypothesis, we assume without loss of generality that $0\leq f\leq 1$.
Using Lemma~\ref{lind0.021} we find a Baire set $B\supset \ext X$ together with a function $g\in\C_\alpha(B)$ with values in $[0,1]$ such that $g=f$ on $\ext X$ and $g(x)=\mu(g)$ for each $x\in B$ and $\mu\in\M_x(X)$ with $\mu(B)=1$.

We claim that $f=g$ on $B$. To verify this, pick $x\in B$ and a maximal measure $\mu\in\M_x(X)$. Then $\mu$ is supported by $B$ and $f=g$ $\mu$-almost everywhere. (Indeed, the set $\{y\in X\colon f(y)=g(y)\}$ is $\mu$-measurable and contains $\ext X$. The assertion thus follows from Lemma~\ref{lind0.020}.) Hence 
\[
g(x)=\mu(g)=\mu(f)=f(x),
\]
where the last equality follows from the strong affinity of $f$. 
This concludes the proof.
\end{proof}

\begin{prop}\label{lind0.03}
Let $X$ be a compact convex set with $\ext X$ Lindel\"of and let $f\colon X\to\er$ be a strongly affine function such that $f\r_{\ext X}$ is Baire.
Then $f$ is a Baire function on $X$.
\end{prop}

\begin{proof}
The assertion follows from Lemmas~\ref{g=f} and~\ref{lind0.02}.
\end{proof}

\section{Transfer of descriptive properties on compact convex sets with $\ext X$ being \lin}
\label{s:transf-lind}

The notions in this section are considered with respect to real numbers. The following key factorization result uses a method of a metrizable reduction available for Baire functions that can be found e.g. in \cite{capon}, \cite[Theorem~5.9.13]{ROJA}, \cite[Theorem~1]{tele2}, \cite{batty2} or \cite[Theorem~9.12]{book-irt}. The main results of Theorem~\ref{lind2} are then consequences of a selection theorem by M.~Talagrand (see \cite{talagr4}).

\begin{lemma}\label{lin-reduk}
Let $X$ be a compact convex set with $\ext X$ Lindel\"of and let $f:X\to\er$ be strongly affine such that $f\r_{\ext X}\in \C_\alpha(\ext X)$ for some $\alpha\in [1,\omega_1)$. Then there exist a metrizable compact convex set $Y$, an affine surjection $\varphi:X\to Y$, a strongly affine Baire function $\wt{f}:Y\to \er$ and $\wt{g}\in\C_\alpha^b(\ext Y)$ such that 
\[
\wt{g}(\varphi(x))=f(x),\quad x\in \ext X\cap \varphi^{-1}(\ext Y),
\]
and
\[
f(x)=\wt{f}(\varphi(x)),\quad x\in X.
\]
\end{lemma}

\begin{proof}
Given a function $f$ as in the premise,  we may assume without loss of generality that $0\leq f\leq 1$. Let $\F =\{g_n\colon n\in\en\}\subset\C(\ext X)$ be a countable family of functions with values in $[0,1]$ satisfying $f\in\F_\alpha$. 

For a fixed index $n\in\en$, using Lemma~\ref{lin-appr} we select finite families $\U_n^k$ and $\L_n^k$, $k\in\en$, of functions in $\fra^c(X)$ with values in $[0,1]$ such that, for 
\[
u_n^k=\inf \U_n^k,\quad l_n^k=\sup \L_n^k,
\]
we have
\begin{itemize}
\item $\lim_{k \to\infty} l_n^k(x)=\lim_{k\to\infty} u_n^k= g_n(x)$ for each $x\in\ext X$,
\item $(l_n^k)_{k=1}^\infty$ is increasing and $(u_n^k)_{k=1}^\infty$ is decreasing.
\end{itemize}

Further, by Proposition~\ref{lind0.03}, $f$ is a Baire function on $X$, say of class $\beta$. Let $\F'=\{h_n\colon n\in\en\}\subset \C(X)$ be a countable family satisfying $f\in (\F')_\beta$. For any $n,k\in\en$, by \cite[Proposition~I.1.1]{alfsen} (or \cite[Proposition~3.11]{book-irt}) there exist finite families $\V_{n}^k,\W_n^k\subset\fra^c(X)$ such that, for $v_n^k=\inf \V_n^k$, $w_n^k=\sup\W_n^k$, we have 
\[
\|h_n-(v_n^k+w_n^k)\|<\frac1k.
\]
By setting $\G=\{v_n^k, w_n^k\colon n,k\in\en\}$, we obtain a family satisfying $f\in \G_\beta$.

We set 
\[
\Phi=\bigcup_{n,k\in\en} \left(\U_n^k\cup \L_n^k\cup \V_n^k\cup \W_n^k\right)
\]
and define
$\varphi:X\to\er^\en$ as
\[
\varphi(x)=\left(\phi(x)\right)_{\phi\in\Phi},\quad x\in X.
\]
Then $Y=\varphi(X)$ is a metrizable compact convex set and, for each $\phi\in\Phi$, there exists $\wt{\phi}\in\fra^c(Y)$ with $\wt{\phi}\circ\varphi=\phi$.

For fixed $n,k\in\en$, let $\wt{\U}_n^k\subset \fra^c(Y)$ be such that 
\[
\U_n^k=\left\{\wt{u}\circ\varphi\colon \wt{u}\in \wt{\U}_n^k\right\}.
\]
Analogously we pick $\wt{\L}_n^k$, $\wt{\V}_{n}^k$ and $\wt{\W}_{n}^k$ in $\fra^c(Y)$. Then 
\[
\wt{u}_n^k=\inf \wt{\U}_n^k,\ \wt{l}_n^k=\sup\wt{\L}_n^k,\ \wt{v}_n^k=\inf \wt{\V}_n^k\quad\text{and}\quad \wt{w}_n^k=\sup \wt{\W}_n^k
\]
satisfy
\[
\wt{u}_n^k\circ\varphi=u_n^k,\ \wt{l}_n^k\circ\varphi=l_n^k,\ \wt{v}_n^k\circ\varphi=v_n^k\quad\text{and}\quad \wt{w}_n^k\circ\varphi=w_n^k.
\]

Given $y\in\ext Y$, we select $x\in \ext X\cap \varphi^{-1}(y)$. Then
\[
\aligned
\lim_{k\to\infty} \wt{u}_n^k(y)&=\lim_{k\to\infty} \wt{u}_n^k(\varphi(x))=\lim_{k\to\infty} u_n^k(x)=g_n(x),\quad\text{and}\\
\lim_{k\to\infty} \wt{l}_n^k(y)&=\lim_{k\to\infty} \wt{l}_n^k(\varphi(x))=\lim_{k\to\infty} l_n^k(x)=g_n(x).
\endaligned
\]
Then $(\wt{u}_n^k)_{k=1}^\infty$ is a decreasing sequence on $\ext Y$, $(\wt{l}_n^k)_{k=1}^\infty$ is increasing on $\ext Y $and both converge to a common limit $\wt{g}_n:\ext Y\to\er$ defined by
\[
\wt{g}_n(y)=\lim_{k\to\infty} \wt{u}_n^k(y),\quad y\in \ext Y.
\]
Then $\wt{g}_n$ is a continuous function on $\ext Y$ with values in $[0,1]$.

Thus, for every $n\in\en$, there exists a function $\wt{g}_n\in\C^b(\ext Y)$ satisfying $\wt{g}_n\circ\varphi=g_n$ on $\ext X\cap \varphi^{-1}(\ext Y)$.
Let $\wt{\F}=\{\wt{g}_n\colon n\in\en\}$.

Now we claim that, for each $\gamma\in [0,\alpha]$ and $h\in\F_\gamma$, there exists $\wt{h}\in \wt{\F}_\gamma$ such that $h=\wt{h}\circ\varphi$ on $\ext X\cap \varphi^{-1}(\ext Y)$. To verify this, we proceed by transfinite induction. The claim is obvious for $\gamma=0$. Assume  that it holds for all $\gamma'<\gamma$ for some $\gamma\leq \alpha$ and that we are given $h\in\F_\gamma$. Let $\gamma_n<\gamma$ and $h_n\in\F_{\gamma_n}$, $n\in\en$, be such that $h=\lim h_n$. By the inductive assumption, there exist $\wt{h}_n\in\wt{\F}_{\gamma_n}$ satisfying $h_n=\wt{h}_n\circ\varphi$ on $\ext X\cap \varphi^{-1}(\ext Y)$. Then the sequence $(\wt{h}_n(y))$ converges for every point $y\in\ext Y$. Hence we may define a function $\wt{h}\in \wt{\F}_\gamma$ by
\[
\wt{h}(y)=\lim_{n\to\infty} \wt{h}_n(y),\quad y\in\ext Y,
\]
and then, for every $y\in\ext Y$ and $x\in\varphi^{-1}(y)\cap \ext X$,
\[
\wt{h}(y)=\lim_{n\to\infty} \wt{h}_n(y)=\lim_{n\to\infty} h_n(x)=h(x).
\]
This proves the claim.

It follows from the claim that there exists a function $\wt{g}\in\C_\alpha(\ext Y)$ such that
\[
\wt{g}(\varphi(x))=f(x),\quad x\in \ext X\cap \varphi^{-1}(\ext Y).
\]

Analogously, let $\wt{\G}$ be the family satisfying 
\[
\G=\{\wt{z}\circ\varphi\colon \wt{z}\in \wt{\G}\}.
\]  
Then, for each $\gamma\in [0,\beta]$ and a function $h\in \G_\gamma$, it follows as above that there exists a function $\wt{h}\in \wt{\G}_\gamma$ satisfying $h=\wt{h}\circ \varphi$. Hence there exists a function $\wt{f}\in (\wt{\G})_\beta$ satisfying $f=\wt{f}\circ \varphi$.
Obviously, $\wt{f}$ is a Baire function and, moreover, it is strongly affine by \cite[Proposition~3.2]{spurny-repre} (see also \cite[Proposition~5.29]{book-irt}).
This concludes the proof.
\end{proof}

\begin{thm}\label{lind2}
Let $\ext X$ be a Lindel\"of set and $f:X\to\er$ be a strongly affine function. If $f\r_{\ext X}\in \C_\alpha(\ext X)$, then 
\[
f\in \begin{cases} \C_{\alpha+1}(X), &\alpha\in [0,\omega_0),\\
                     \C_\alpha(X), &\alpha\in [\omega_0,\omega_1).
      \end{cases}
\]   
\end{thm}

\begin{proof}
Let $f$ be a strongly affine function $f$ whose restriction to $\ext X$ is of Baire class $\alpha$. If $\alpha=0$, i.e., $f$ is continuous and bounded on $\ext X$,  Lemma~\ref{lin-appr} provides the relevant sequences $(u_n)$ and $(l_n)$. For $n\in\en$, $x\in X$ and $\mu_1,\mu_2\in\M_x(X)$, we have
\[
\mu_1(l_n)\leq \mu_1(f)=f(x)=\mu_2(f)\leq \mu_2(u_n).
\]
By \cite[Corollary~I.3.6]{alfsen} (see also \cite[Lemma~3.21]{book-irt}), 
\[
(l_n)^*\leq f\leq (u_n)_*.
\]
By the Hahn-Banach theorem, there exists a sequence $(h_n)$ of functions in $\fra^c(X)$ such that 
\[
(l_n)^*-\frac1n<h_n<(u_n)_*+\frac1n,\quad n\in\en.
\]
Then $f\in\C_1(X)$ because $h_n\to f$ on $\ext X$, and thus on $X$. (Indeed, given $x\in X$, let $\mu\in\M_x(X)$ be maximal. Then the set 
\[
\{y\in X\colon h_n(y)\to f(y)\}
\]
is  $\mu$-measurable and contains $\ext X$. By Lemma~\ref{lind0.020}, $\mu(B)=1$. Hence $f(x)=\mu(f)=\lim \mu(h_n)=h_n(x)$.)

Assume now that $\alpha\geq 1$. Then we use Lemma~\ref{lin-reduk} to find a continuous affine surjection $\varphi$ of $X$ onto a metrizable compact convex set $Y$, $\wt{g}\in\C_\alpha^b(\ext Y)$ and a Baire function $\wt{f}:X\to\er$ such that
\begin{equation}\label{comp}
f=\wt{g}\circ\varphi\text{ on }\ext X\cap\varphi^{-1}(\ext Y)\quad\text{and}\quad f=\wt{f}\circ\varphi\text{ on }X.
\end{equation}

Since $\ext Y$ is  a $G_\delta$ set and $\alpha\geq 1$, we can extend $\wt{g}$ to the whole set $Y$ (and denote it likewise) with preservation of class (see  \cite[\S\,31, VI, Th\'eor\`eme]{kur}). By \cite[Th\'eor\'eme~1]{talagr4} (see also \cite[Theorem~11.41]{book-irt}), there exists a mapping $y\mapsto\nu_y$, $y\in Y$, such that
\begin{enumerate}
\item [(a)] $\nu_y$ is a maximal measure in $\M_y(Y)$,
\item [(b)] the function $y\mapsto \nu_y(h)$ is Baire-one on $Y$ for every $h\in\C(Y)$.
\end{enumerate}
Let
\[
\wt{h}(y)=\nu_y(\wt{g}),\quad y\in Y.
\]
Then 
\begin{equation*}
\wt{h}\in \begin{cases} \C_{\alpha+1}(Y), &\alpha\in [0,\omega_0),\\
                     \C_\alpha(Y), &\alpha\in [\omega_0,\omega_1).
      \end{cases}
\end{equation*}
Indeed, if $\alpha<\omega_0$, the claim follows from (b) by induction. If $\alpha=\omega_0$, let $(\wt{g}_n)$ be a bounded sequence of functions such that
$\wt{g}_n\in\C_{\alpha_n}(Y)$ for some $\alpha_n<\omega_0$ and $\wt{g}_n\to \wt{g}$. Then the functions $\wt{h}_n(y)=\nu_y(\wt{g}_n)$ are in $\C_{\alpha_n+1}(Y)$ and converge to $\wt{h}$. Hence $\wt{h}\in \C_{\omega_0}(Y)$. For $\alpha>\omega_0$, the claim follows by transfinite induction.

Next we prove that $\wt{h}=\wt{f}$. To this end, let $y\in Y$ be fixed. Using \cite[Proposition~7.49]{book-irt} we find a maximal measure $\mu\in\M^1(X)$ satisfying $\varphi_\sharp\mu=\nu_y$ (here $\varphi_\sharp:\M^1(X)\to \M^1(Y)$ denotes the mapping induced by $\varphi:X\to Y$, see \cite[Theorem~418I ]{fremlin4}). Then it is easy to check (see e.g. the proof of Proposition~5.29 in \cite{book-irt}) that
\begin{equation}\label{p-bar}
\varphi(r(\mu))=r(\varphi_\sharp\mu)=r(\nu_y)=y.
\end{equation}
Further, 
\[
\mu(\varphi^{-1}(\ext Y))=1
\]
and 
\[
\{x\in X\colon f(x)=\wt{g}(\varphi(x))\}\supset \ext X\cap \varphi^{-1}(\ext Y).
\]
From these facts and Lemma~\ref{lind0.020} it follows that $f=\wt{g}\circ\varphi$ $\mu$-almost everywhere.
Thus we get from \eqref{p-bar} and \eqref{comp}
\[
\aligned
\wt{h}(y)&=\int_{\ext Y} \wt{g}\, d\nu_y=\int_{\ext Y} \wt{g}\, d(\varphi_\sharp\mu)\\
&=\int_X \wt{g}\circ\varphi\, d\mu=\int_X f\, d\mu\\
&=f(r(\mu))=\wt{f}(\varphi(r(\mu)))\\
&=\wt{f}(y).
\endaligned
\]
Hence $\wt{f}=\wt{h}$ on $Y$.

By \eqref{comp}, $f$ is of the same class as $\wt{f}=\wt{h}$. This concludes the proof.
\end{proof}

\section{Transfer of decriptive properties on compact convex sets with $\ext X$ being a resolvable \lin\ set}
\label{s:transf-resol}

Again we point out that this section works within the context of real spaces. The first important ingredient is a result on separation of \lin\ sets in Tychonoff spaces. 

\begin{lemma}\label{separ}
Let $X_1$ and $X_2$ be disjoint \lin\ sets in a Tychonoff space $X$. Assume that there is no set $G\subset X$ satisfying  $X_1\subset G\subset X\setminus X_2$ which is a countable intersection of cozero sets. Then there exists a nonempty closed set $H\subset X$ with $\ov{H\cap X_1}=\ov{H\cap X_2}=H$.
\end{lemma}

\begin{proof}
See \cite[Proposition~11]{KaSp}.
\end{proof}

The following lemma is a kind of a selection result.

\begin{lemma}\label{selec}
Let $\varphi:X\to Y$ be a continuous surjective mapping of a compact space $X$ onto a compact space $Y$ and let $f:X\to\er$ be a bounded $\Sigma_{\alpha}(\Bos(X))$-measurable function for some $\alpha\in [2,\omega_1)$. Then there exists a mapping $\phi:Y\to X$ such that
\begin{itemize}
\item $\varphi(\phi(y))=y$, $y\in Y$,
\item $f\circ\phi$ is a $\Sigma_{\alpha}(\Bos(Y))$-measurable function.
\end{itemize}
\end{lemma}

\begin{proof}
Given a bounded $\Sigma_{\alpha}(\Bos(X))$-measurable function $f$ on $X$, we construct using a standard approximation technique and \cite[Proposition~2.3(f)]{spurny-ahm} (see also \cite[Lemma~5.7]{book-irt}) a bounded sequence $(f_n)$ of $\Sigma_{\alpha}(\Bos(X))$-measurable simple functions uniformly converging to $f$. More precisely, each $f_n$ is of the form 
\[
f_n=\sum_{k=1}^{k_n} c_{nk} \chi_{A_{nk}},\quad c_{nk}\in\er, A_{nk}\in \Delta_\alpha(\Bos(X))\text{ for }k=1,\dots, k_n,
\]
where the family $\{A_{nk}:k=1,\dots,k_n\}$ is a disjoint cover of $X$.
For every set $A_{nk}$ we consider a countable family $\A_{nk}\subset \Bos(X)$ satisfying $A_{nk}\in \Sigma_{\alpha}(\A_{nk})$.
We include all these families in a single family $\A$.

By \cite[Lemma~8]{HoSp}, there exists a mapping $\phi:Y\to X$ such that $\varphi(\phi(y))=y$ for every $y\in Y$ and $\phi^{-1}(A)\in \Bos(Y)$ for every $A\in \A$. Then both $\phi^{-1}(A_{nk})$ and $\phi^{-1}(X\setminus A_{nk})$ are in $\Sigma_{\alpha}(\Bos(Y))$ for every set $A_{nk}$. Thus the functions $f_n\circ\phi$ are $\Sigma_{\alpha}(\Bos(Y))$-measurable and consequently, since they converge uniformly to $f\circ\phi$, the function $f\circ\phi$ is $\Sigma_{\alpha}(\Bos(Y))$-measurable as well.
\end{proof}

The next assertion provides an inductive step needed in the proof of Theorem~\ref{porce}.

\begin{lemma}\label{partU}
Let $X$ be a compact convex set with $\ext X$ being a resolvable \lin\ set and $f:X\to\er$ be a strongly affine function such that $f\r_{\ext X}\in \C_\alpha(\ext X)$ for some $\alpha\in [1,\omega_0)$. Let $K\subset X$ be a nonempty compact set and $\ep>0$. Then there exists a nonempty open set $U$ in $K$ and a $\Sigma_{\alpha+1}(\Hs(U))$-measurable function $g$ on $U$ such that $|g-f|<\ep$ on $U$.
\end{lemma}

\begin{proof}
Without loss of generality we assume that $0\leq f\leq 1$.
Let $K$ be a compact set in $X$ and $\ep>0$. By Lemma~\ref{g=f}, there exists a Baire set $B\supset \ext X$ such that $f\in\C_\alpha(B)$.
We claim that there exists a $G_\delta$ set $G$ with
\begin{equation}\label{sepG}
X\setminus B\subset G\subset X\setminus \ext X.
\end{equation}
Indeed, if there were no such set, Lemma~\ref{separ} applied to $X_1=X\setminus B$ and $X_2=\ext X$ (observe that $X\setminus B$ is  \lin\ since it is a Baire set; see \cite[Theorem~2.7.1]{ROJA}) would provide a nonempty closed set $H\subset X$ satisfying  $\ov{H\cap (X\setminus B)}=\ov{H\cap \ext X}=H$. But this would contradict the fact that $\ext X$ is a resolvable set.

We pick a $G_\delta$ set $G$ satisfying~\eqref{sepG} and write $F=X\setminus G=\bigcup F_n$, where the sets $F_1\subset F_2\subset\cdots$ are closed in $X$. Then $\ext X\subset \bigcup F_n\subset B$.

For each $n\in\en$, we set
\[
\aligned
M_n&=\{\mu\in\M^1(X)\colon \mu(F_n)\geq 1-\ep\}\quad\text{and}\\
X_n&=\{x\in X\colon \text{there exists }\mu\in M_n\text{ such that }r(\mu)=x\}\ (=r(M_n)).
\endaligned
\]
Then each $X_n$ is a closed set by the upper semicontinuity of the function $\mu\mapsto \mu(F_n)$ on $\M^1(X)$ 
and $X=\bigcup X_n$. Indeed, for any $x\in X$ there exists a maximal measure $\mu\in\M_x(X)$, which is carried by $F$ (see \cite[Corollary~I.4.12 and the subsequent remark]{alfsen} or \cite[Theorem~3.79]{book-irt}), and thus $\mu(F_n)\geq 1-\ep$ for $n\in\en$ large enough. 

Since $K\subset \bigcup X_n$, by the Baire category theorem there exists $m\in\en$ such that $X_m\cap K$ has nonempty interior in $K$. 
Let $U$ denote this interior. Since $f\r_{F_m}\in\C_\alpha(F_m)$, we can extend $f\r_{F_m}$ to a function $h\in\C_\alpha(X)$ satisfying $h(X)\subset \ov{\co} f(F_m)$ (see \cite[Corollary~3.5]{spurnyweak} or \cite[Corollary~11.25]{book-irt}). Let the functions $\wt{h}, \wt{f}\colon \M^1(X)\to \er$ be defined as 
\[
\wt{h}(\mu)=\mu(h),\ \wt{f}(\mu)=\mu(f),\quad \mu\in\M^1(X).
\]
Then 
\begin{equation}\label{fh-ep}
|\wt{f}(\mu)-\wt{h}(\mu)|<\ep,\quad \mu\in M_m.
\end{equation}
By Lemma~\ref{transM}(c), $\wt{h}\in\C_\alpha(\M^1(X))$, and thus it is $\Sigma_{\alpha+1}(\Bos(\M^1(X)))$-measurable on $\M^1(X)$.

We consider the mapping $r:M_m\to r(M_m)$ and use Lemma~\ref{selec} to find a selection $\phi: r(M_m)\to M_m$ such that 
\begin{itemize}
\item $r(\phi(x))=x$, $x\in r(M_m)$,
\item $\wt{h}\circ\phi$ is $\Sigma_{\alpha+1}(\Bos(r(M_m)))$-measurable on $r(M_m)$.
\end{itemize}

By setting $g=\wt{h}\circ\phi$ we obtain the desired function. Indeed, for a given point $x\in r(M_m)$, the measure $\phi(x)$ is contained in $\M_x(X)\cap M_m$, and hence by \eqref{fh-ep} and the strong affinity of $f$, we have
\[
|g(x)-f(x)|=|\wt{h}(\phi(x))-\wt{f}(\phi(x))|<\ep.
\]
Thus the function $g\r_{U}$ is the required one because $\Sigma_{\alpha+1}(\Bos)$-measurability implies  $\Sigma_{\alpha+1}(\Hs)$-measurability.
\end{proof}

\begin{thm}
\label{porce}
Let $X$ be a compact convex set with $\ext X$ being a resolvable \lin\ set. Let $f:X\to\er$ be a strongly affine function such that $f\r_{\ext X}\in \C_\alpha(\ext X)$ for some $\alpha\in [1,\omega_1)$. Then $f\in\C_\alpha(X)$.
\end{thm}

\begin{proof}
Given such a function $f$, we assume that $0\leq f\leq 1$. Also we may assume that $\alpha\in [1,\omega_0)$ since other cases are covered by  Theorem~\ref{lind2}.
We claim that $f$ is $\Sigma_{\alpha+1}(\Hs(X))$-measurable.

To this end, let $\ep>0$ be arbitrary. We construct a regular sequence $\emptyset=U_0\subset U_1\subset\cdots\subset U_\kappa=X$ and functions 
\[
g_\gamma\in \Sigma_{\alpha+1}(\Hs(U_{\gamma+1}\setminus U_\gamma)),\quad \gamma<\kappa,
\]
satisfying $|g-f|<\ep$ on $U_{\gamma+1}\setminus U_\gamma$ as follows. 

Let $U_0=\emptyset$. Using Lemma~\ref{partU} we select a nonempty open set $U$ of $X$ along with a $\Sigma_{\alpha+1}(\Hs(U)$-measurable function $g$ on $U$ with $|g-f|<\ep$ on $U$. We set $U_1=U$ and $g_0=g$.

Assume now that $U_\delta$ and $g_\delta$ are chosen for all $\delta$ less then some $\gamma$. If $\gamma$ is limit, we set $U_\gamma=\bigcup_{\delta<\gamma} U_\delta$. 

Let $\gamma=\lambda+1$. If $U_\lambda=X$, we set $\kappa=\lambda$ and stop the procedure. Otherwise we apply Lemma~\ref{partU} to $K=X\setminus U_\lambda$ and obtain an open set $U\subset X$ intersecting $K$ along with a
$\Sigma_{\alpha +1}(\Hs(U\cap K))$-measurable function $g$ on $U\cap K$ satisfying $|g-f|<\ep$ on $U\cap K$. We set $U_\gamma=U_\lambda\cup U$ and $g_\lambda=g$. This finishes the construction.

Let $g:X\to \er$ be defined as $g=g_\gamma$ on $U_{\gamma+1}\setminus U_\gamma$, $\gamma<\kappa$. By Proposition~\ref{propH}, $g$ is a $\Sigma_{\alpha+1}(\Hs(X))$-measurable function.

By the procedure above we can approximate uniformly $f$ by $\Sigma_{\alpha+1}(\Hs(X))$-measurable functions which yields that $f$ itself is $\Sigma_{\alpha+1}(\Hs(X))$-measurable. Since $f$ is a Baire function by Proposition~\ref{lind0.03}, Theorem~5.2 and Corollary~5.5 in \cite{spurny-ahm} gives $f\in\C_\alpha(X)$. This finishes the proof.
\end{proof}

\section{Proofs of the main results}
\label{s:main}

Before proving main results we recall a simple observation.

\begin{lemma}\label{ce-er}
Let $E$ be a complex Banach space and let $f\in E^{**}$. Then $f$ is strongly affine on $B_{E^*}$ if and only if $\re f$ is strongly affine  on $B_{E^*}$.
\end{lemma}

\begin{proof}
If $f$ is strongly affine on $B_{E^*}$ and $\mu\in\M^1(B_{E^*})$ has $x^*$ as its barycenter, then 
\[
\re f(x^*)+i\im f(x^*)=f(x^*)=\mu(f)=\mu(\re f)+i\mu(\im f),
\]
and thus $\mu(\re f)=\re f(x^*)$ and $\mu(\im f)=\im f(x^*)$.

Conversely, assuming that $\re f$ is strongly affine on $B_{E^*}$, we infer that so is $\im f$. 
To see this, consider an affine surjective homeomorphic mapping $\varphi: B_{E^*}\to B_{E^*}$ defined as 
\[
\varphi(y^*)=iy^*,\quad y^*\in B_{E^*}.
\]
Since $\im f(y^*)=-\re f(iy^*)$ for $y^*\in E^*$, the function $\im f$ is a composition of an affine homeomorphism and a strongly affine function, and hence it is strongly affine as well.
Thus, for $\mu\in\M^1(B_{E^*})$ with the barycenter $x^*$, 
\[
\mu(f)=\mu(\re f)+i\mu(\im f)=\re f(x^*)+i \im f(x^*)=f(x^*),
\]
and $f$ is strongly affine.
\end{proof}

\begin{proof}[Proofs of Theorems~\ref{main1},~\ref{main2} and~\ref{main3}]
We proceed to the proofs of Theorems~\ref{main1},~\ref{main2} and~\ref{main3}. Let $E$ be a (real or complex) Banach space and $f$ be an element of $E^{**}$ whose restriction to $B_{E^*}$ is strongly affine. By forgetting in $E^*$ the multiplication by complex numbers, we can regard $B_{E^*}$ to be a compact convex set in a real locally convex space. The function $\re f$ is then a strongly affine function on a compact convex set $B_{E^*}$ that inherits all descriptive properties from $f$. Thus if $f\r_{\ov{\ext B_{E^*}}}\in \Hf_\alpha(\ov{\ext B_{E^*}})$, then $\re f$ is a strongly affine real-valued function with $\re f\r_{\ov{\ext B_{E^*}}}\in \Hf_\alpha(\ov{\ext B_{E^*}})$. An application of Theorem~\ref{transf-cl} gives $\re f\in \Hf_\alpha(B_{E^*})$. Then both $\re f$ and $\im f$ are in $\Hf_\alpha(B_{E^*})$, and thus $f=\re f+i \im f$ is in $\Hf_\alpha(B_{E^*})$. Similarly we prove the other assertions of Theorem~\ref{main1}.

Apparently, this procedure also verifies Theorems~\ref{main2} and~\ref{main3}, which finishes their proof.
\end{proof}

\begin{proof}[Proof of Theorem~\ref{main4}]
Now we prove Theorem~\ref{main4}. From now on we will be working with real spaces. We start with the following assertion which shows the required result for Banach spaces of continuous affine functions on simplices. The general result will be then obtained by means of a result of W.~Lusky in \cite{lusky5}.

\begin{prop}\label{simp}
Let $f:X\to\er$ be a strongly affine function on a simplex $X$ such that $f\in\C_\alpha(X)$ for some $\alpha\geq 2$. Then
\[
f\in\begin{cases} \fra_{\alpha+1}(X),& \alpha\in [2,\omega_0),\\
                  \fra_\alpha(X),&\alpha\in [\omega_0,\omega_1).
      \end{cases}                          
\]
If, moreover, $\ext X$ is a \lin\ resolvable set, then $f\in \fra_{\alpha}(X)$.
\end{prop}

\begin{proof}[Proof of Proposition~\ref{simp}]
If $X$ is a general simplex, the assertion for finite ordinals is proved in \cite[Th\'eor\`eme]{capon}, for infinite ordinals in \cite[Theorem~1.2]{kacspu}.

Assume now that $X$ is a simplex with $\ext X$ being a \lin\ resolvable set.
For each $x\in X$, let $\delta_x$ denote the unique maximal measure in $\M_x(X)$. By \cite[Theorem~1]{KaSp2}, the function $Tg(x)=\delta_x(g)$, $x\in X$, is in $\fra_1(X)$ for any bounded $g\in\C_1(X)$. By induction, $Tg\in\fra_\beta(X)$ for any bounded function $g\in\C_\beta(X)$ and finite ordinal $\beta\in [2,\omega_0)$. 
Thus, for any $\alpha\in [2,\omega_0)$ and a strongly affine function $f\in \C_\alpha(X)$, $f=Tf\in\fra_\alpha(X)$. This finishes the proof.
\end{proof}

Let $E$ be a real $L_1$-predual and $f\in E^{**}$ be a strongly affine function satisfying $f\in\C_\alpha(B_{E^*})$ for some $\alpha\in [2,\omega_1)$.
By \cite[Theorem]{lusky5}, there exist a simplex $X$, an isometric embedding $j: E\to \fra^c(X)$ and a projection $P:\fra^c(X)\to j(E)$ of norm $1$.  Further, it is proved in \cite[Corollary III]{lusky5} that there exists an affine continuous surjection $\varphi:X\to B_{E^*}$ such that
\begin{enumerate}
\item [\rm(1)] $\varphi(\ext X)=\ext B_{E^*}\cup\{0\}$ and $\varphi^{-1}(\ext B_{E^*})\subset \ext X$,
\item [\rm(2)] $\varphi\r_{\ext X}$ is injective,
\item [\rm(3)] $\ext X\setminus \varphi^{-1}(\ext B_{E^*})$ is a singleton,
\item [\rm(4)] $j(e)(x)=(e\circ\varphi)(x)$, $e\in E$, $x\in X$.
\end{enumerate}
(In the notation of \cite{lusky5}, the embedding $j$ is denoted by $T$ and $\varphi$ is denoted by $q$. Conditions (1), (2) and (3) are explicitly stated in \cite[Corollary III]{lusky5}, condition (4) follows from the definitions of $T$ on p.\,175 and $q$ on p.\,176.)

The projection $P$ provides for each $x\in X$ a measure $\mu_x\in B_{\M(X)}$ such that 
\begin{equation}
\label{proje}
Pg(x)=\mu_x(g),\quad g\in\fra^c(X).
\end{equation}
Since $P$ is identity on $j(E)$, we obtain from (4)
\[
\mu_x(e\circ\varphi)=(e\circ\varphi)(x),\quad x\in X, e\in E.
\]
We use equality \eqref{proje} to extend the domain of $P$ to any bounded universally measurable function on $X$.

We claim that
\begin{equation}
\label{hacko}
\mu_x(f\circ\varphi)=f(\varphi(x)),\quad x\in X.
\end{equation}
To verify this, let $x\in X$ be given. We write
\[
\mu_x=a_1\mu_1-a_2\mu_2,\quad a_1,a_2\geq 0\text{ with }a_1+a_2\leq 1,\ \mu_1,\mu_2\in \M^1(X),
\]
and let $x_1, x_2\in X$ be the barycenters of $\mu_1$ and $\mu_2$, respectively.
Then
\begin{equation}
\label{ixko}
\varphi(x)=a_1\varphi(x_1)-a_2\varphi(x_2).
\end{equation}
Indeed, let $e\in E$ be arbitrary.
The  we compute 
\[
\aligned
e(\varphi(x))&=\mu_x(e\circ\varphi)=a_1\mu_1(e\circ\varphi)-a_2\mu_2(e\circ\varphi)\\
&=a_1 e(\varphi(x_1))-a_2 e(\varphi(x_2))\\
&=e(a_1\varphi(x_1)-a_2\varphi(x_2)).
\endaligned
\]
Hence \eqref{ixko} holds.

Since $f\circ\varphi$ is strongly affine on $X$ by \cite[Lemma~2.3]{spurny-trans} (see also \cite[Proposition~5.29]{book-irt}), we get from \eqref{ixko}
\[
\aligned
\mu_x(f\circ\varphi)&=a_1\mu_1(f\circ\varphi)-a_2\mu_2(f\circ\varphi)=a_1 f(\varphi(x_1))-a_2 f(\varphi(x_2))\\
&=f(a_1\varphi(x_1)-a_2\varphi(x_2))=f(\varphi(x)).
\endaligned
\]
This verifies \eqref{hacko}.

Now we prove by induction that $Pg\in (j(E))_\beta$ provided $g\in \fra_\beta(X)$ for some $\beta\geq 1$.
First consider the case $\beta=1$, i.e., there exists a bounded sequence $(g_n)$ in $\fra^c(X)$ with $g_n\to g$.
Then $Pg_n\in j(E)$ and, by the Lebesgue dominated convergence theorem, $Pg_n\to Pg$.

Assuming the validity of the assertion for all ordinals $\wt{\beta}$ smaller then some $\beta$, we consider $g\in \fra_\beta(X)$. Let $(g_n)$ be a bounded sequence  converging pointwise to $g$, where $g_n\in\fra_{\beta_n}(X)$ for some $\beta_n<\beta$. Then $Pg_n\in (j(E))_{\beta_n}$ and, as above, $Pg_n\to Pg$.

Now we get back to the function $f$. Since $f\circ\varphi\in\C_\alpha(X)$, Proposition~\ref{simp} implies that the function $f\circ\varphi$ belongs to $\fra_\beta(X)$, where either $\beta=\alpha+1$ if $\alpha<\omega_0$ or $\beta=\alpha$ otherwise.
By the reasoning above and~\eqref{hacko}, 
\[
f\circ\varphi=P(f\circ\varphi)\in (j(E))_{\beta}.
\]
Since $j(e)=e\circ\varphi$ for each $e\in E$, it follows that $f\in \fra_\beta(B_{E^*})$.
This concludes the proof of the first part of the theorem.

If, moreover, we assume that $\ext B_{E^*}$ is a \lin\ resolvable set, we observe that $\ext X$ is a \lin\ resolvable set as well. 
To show this, we first notice that $\ext X$ differs from the resolvable set  $\varphi^{-1}(\ext B_{E^*})$ by a singleton (see (1) and (3)), and thus it is a resolvable set. Second, let $F\subset X\setminus \ext X$ be a compact set. By (1), $\varphi(F)$ is disjoint from $\ext B_{E^*}$. 
Since $\ext B_{E^*}$ is Lindel\"of, \cite[Lemma 14]{KaSp2} provides an $F_\sigma$ set $A$ with 
\[
\ext B_{E^*}\subset A\subset  B_{E^*}\setminus \varphi(F).
\]
If $x_0\in X$ denotes the singleton $\ext X\setminus \varphi^{-1}(\ext B_{E^*})$, then
$\varphi^{-1}(A)$ is an $F_\sigma$ set in $X$ satisfying
\[
\ext X\subset \varphi^{-1}(A)\cup \{x_0\}\subset X\setminus F.
\]
By \cite[Lemma 15]{KaSp2}, $\ext X$ is a Lindel\"of space.

Now we can conclude the proof as in the first part, the only difference is that we use the second part of Proposition~\ref{simp}.
\end{proof}

\section{Examples}
\label{examples}

Banach spaces constructed in this section are real $L_1$-preduals and they are created using a notion of a \emph{simplicial function space}. In order to illuminate the construction, we need to recall several definitions and facts. 

If $K$ is a compact topological space, $\H\subset\C(K)$ is a \emph{function space} if $\H$ is a subspace of $\C(K)$, contains constant functions and separate points of $K$. For the sake of simplicity, we will construct real Banach spaces, and thus we will deal in this section only with real spaces $\C(K)$. For $x\in K$, we write $\M_x(\H)$ for the set of all measures $\mu\in\M^1(K)$ with $\mu(h)=h(x)$ for all $h\in\H$. Let $\Ch_\H(K)$ be the \emph{Choquet boundary} of $\H$, i.e., the set of those points $x\in K$ with $\M_x(\H)=\{\ep_x\}$. By defining $\A^c(\H)=\{f\in\C(K)\colon \mu(f)=f(x), x\in K, \mu\in\M_x(\H)\}$ we obtain a closed function space satisfying $\H\subset \A^c(\H)$ (see \cite[Definition~3.8]{book-irt}) and $\Ch_{\H}(K)=\Ch_{\A^c(\H)}(K)$ (this follows easily from the definitions).

Let 
\[
\es(\H)=\{s\in \H^*\colon s\geq 0, \|s\|=1\}
\]
denote the \emph{state space} of $\H$. Then $\es(\H)$, endowed with the weak* topology, is a compact convex set and $K$ is homeomorphically embedded in $\es(\H)$ via the mapping $\phi:K\to\es(\H)$ assigning to each $x\in K$ the point evaluation at $x$. Moreover, $\phi(\Ch_\H(K))=\ext \es(\H)$ (see \cite[Proposition~6.2]{PHE} or \cite[Proposition~4.26]{book-irt}).

The function space $\H$ is called \emph{simplicial} if $\es(\A^c(\H))$ is a simplex (see \cite[Theorem~6.54]{book-irt}). 

Further, let $\H^{\perp\perp}$ denote the space of all universally measurable functions  $f:K\to\er$ satisfying $\mu(f)=0$ for every $\mu\in \H^{\perp}\subset \M(K)$. It is proved in \cite[Theorem~2.5]{spurny-trans} (see also \cite[Corollary~5.41]{book-irt}) that for any function $f\in\H^{\perp\perp}$ there exists a strongly affine function $\wt{f}:\es(\H)\to\er$ with $f=\wt{f}\circ \phi$. Moreover, the function $\wt{f}$ inherits from $f$ all descriptive properties considered in the paper, precisely, for any $\alpha\in [1,\omega_1)$ we have $f\in \C_\alpha(K)$, $f\in\Bof_\alpha(K)$ and $f\in\Hf_\alpha(K)$ if and only if $\wt{f}\in \C_\alpha(\es(\H))$, $\wt{f}\in\Bof_\alpha(\es(\H))$ and $\wt{f}\in\Hf_\alpha(\es(\H))$, respectively (the first two assertions are proved in \cite[Corollary~5.41]{book-irt}, the last one follows from Theorem~\ref{transf-cl}).

A standard construction from \cite[Section VII]{BdL} of a simplicial function space $\H$ satisfying $\H=\A^c(\H)$ goes as follows. Take a compact space $L$, its subset $B\subset L$ and define 
\[
K=(L\times\{0\})\cup (B\times \{-1,1\})
\]
with the ``porcupine topology'', i.e., points of $K\setminus (L\times \{0\})$ are discrete and a point $(x,0)\in K$ has a basis of neighborhoods consisting of sets of the form 
\[
K\cap(U\times \{-1,0,1\})\setminus F,
\]
where $U\subset L$ is a neighborhood of $x$ and $F\subset K\setminus (L\times\{0\})$ is finite.
Then $K$ is a compact space and
\[
\H=\{f\in\C(K)\colon f(x,0)=\frac12(f(x,1)+f(x,-1)), x\in B\}
\]
is a simplicial function space satisfying $\H=\A^c(\H)$ and 
\[
\Ch_\H(K)=K\setminus (B\times\{0\})
\]
(for the verifications of these facts see \cite{stacey2} or \cite[Definition~6.13 and Lemma~6.14]{book-irt}).

If $f:K\to\er$ is a bounded universally measurable function satisfying $f(x,0)=\frac12(f(x,1)+f(x,-1))$ for each $x\in B$, it is easy to verify that $f\in \H^{\perp\perp}$ (see \cite[Corollary~6.12]{book-irt}), and thus it induces a strongly affine function $\wt{f}:\es(\H)\to\er$ which satisfies $f=\wt{f}\circ\phi$ and shares with $f$ all descriptive properties.

By this procedure we obtain a simplex $X=\es(\H)$ and a strongly affine function on $X$ with the desired descriptive properties. 
It is well known (see e.g. \cite[Propositions~4.31 and~4.32]{book-irt}) that, given a compact convex set $X$, the dual space $(\fra^c(X))^*$ can be identified with $\span X$ and the dual unit ball with $\co (X\cup (-X))$, whereas the second dual $(\fra^c(X))^{**}$ equals to the space of all affine bounded functions on $X$.
Hence the construction of a simplex $X$ along with a strongly affine function $f$ with the prescribed descriptive properties yields the resulting $L_1$-predual $E$: we set $E=\fra^c(X)$ and the element $x^{**}\in E^{**}$ is the function $f$.

This general construction is now used in the following examples.

\begin{example}\label{noFsigma}
There exist a separable $L_1$-predual $E$ and a strongly affine function $f\in E^{**}$ such that $f\r_{\ext B_{E^*}}\in \C_1(\ext B_{E^*})$ and $f\notin\C_1(B_{E^*})$.
\end{example}

\begin{proof}
Let $L=[0,1]$ and $B$ denote the set of all rational numbers in $L$. 
Let $K$, $\H$ and $X$ be constructed as above. Then $K$ is metrizable, and thus $E=\fra^c(X)$ is a separable space. Let $f:K\to \er$ be defined as
\[
f(x,t)=\begin{cases} 1,& x\in B,\\
                     0,& x\notin B,
       \end{cases}
       \quad (x,t)\in K.
\]
Then $f\r_{\Ch_\H(K)}\in \C_1(\Ch_\H(K))$ since $f\r_{\Ch_\H(K)}$ is the characteristic function of an open set in $\Ch_\H(K)$. On the other hand, $f$ has no point of continuity on $L\times\{0\}$, and thus $f\notin \C_1(K)$. 
\end{proof}

\begin{example}\label{ex1}
There exist an $L_1$-predual $E$ and  a strongly affine function $f\in E^{**}$ such that 
$\ext B_{E^*}$ is an open set in $\ov{\ext B_{E^*}}$ (hence $\ext B_{E^*}\in \Bos(B_{E^*})$), $f\r_{B_{E^*}}\in \C(\ext B_{E^*})$ and $f$ is not resolvably measurable on $B_{E^*}$.
\end{example}

\begin{proof}
Let $L=B=[0,1]$ and $A$ be an analytic non-Borel set in $L$ (see \cite[Theorem~14.2]{kechris}) and let $K$, $\H$ and $X$ be constructed as above.
Then $\Ch_\H(K)=K\setminus (L\times\{0\})$ is an open set in $\ov{\Ch_\H(K)}=K$.
Further, let $f:K\to \er$ be defined as
\[
f(x,t)=\begin{cases} 1,& x\in A,\\
                     0,& x\notin A,
       \end{cases}
       \quad (x,t)\in K.
\]
Then $f\r_{\Ch_\H(K)}\in \C(\Ch_\H(K))$ since $f\r_{\Ch_\H(K)}$ is the characteristic function of a clopen set in $\Ch_\H(K)$. Since $A$ is $\mu$-measurable for any Radon measure $\mu$ on $[0,1]$, $f$ is universally measurable on $K$ (see \cite[Theorem~21.10]{kechris}).  Obviously, $f\r_{L\times\{0\}}$ is not Borel on $L\times\{0\}$. Since the $\sigma$-algebra of Borel sets in $L$ coincides with the $\sigma$-algebra generated by resolvable sets in $L$ (see \cite[Proposition~3.4]{spurny-ahm}), $f$ is not measurable on $K$ with respect to the $\sigma$-algebra generated by resolvable sets.
\end{proof}

\begin{example}\label{ex2}
Assuming (CH), there exist an $L_1$-predual $E$  with $\ext B_{E^*}$ Lindel\"of and a strongly affine function $f\in E^{**}$ such that $f\r_{\ext B_{E^*}}\in \Bof_1(\ext B_{E^*})$ and $f$ is not a resolvably measurable function. 
\end{example}

\begin{proof}
Let $L=[0,1]$ and $Q$ stand for the set of all rational numbers in $L$.
Assuming the continuum hypothesis, by the method of the proof of \cite[Proposition~4.9]{MoRe} we construct an uncountable set $B$ disjoint from
 $Q$ that concentrates around the set $Q$ (i.e.,  the set $B\setminus U$ is countable for any open set $U\supset Q$). Let $K$, $\H$ and $X$ be as above.
Then $\Ch_\H(K)=K\setminus (B\times\{0\})$ is \lin. Indeed, if $\U$ is an open cover of $\Ch_\H(K)$, we select a countable family $\V\subset \U$ satisfying 
\[
(L\times\{0\})\setminus(B\times\{0\})\subset V=\bigcup\{U\cap (L\times\{0\})\colon U\in \V\}.
\]
Then $V$ is an open set in $L\times\{0\}$ containing $Q\times\{0\}$, and thus $B\setminus V$ is countable. Hence we may extract a countable family $\W\subset\U$ which covers that part of $\Ch_\H(K)$ not already contained in $V$. Thus $\V\cup \W$ is a countable subcover of $\Ch_\H(K)$.
 
Define a function $f:K\to\er$ by the formula
\[
f(x,t)=\begin{cases} 1,& x\in B,\\
                     0,& x\notin B,
       \end{cases}
       \quad (x,t)\in K.
\]
Then $f$ is universally measurable on $K$. To see this, it is enough to verify that $B$ is universally measurable. If $\mu\in\M^1([0,1])$ is a continuous measure (i.e., $\mu(\{x\})=0$ for each $x\in [0,1]$), let $(U_n)$ be a sequence of open sets satisfying $\mu(U_n)<\frac1n$ and $U_n\supset Q$. Then $\mu(\bigcap U_n)=0$ and $B\setminus \bigcap U_n$ is countable, and thus $\mu$-measurable. Hence $B$ is $\mu$-measurable for every continuous measure. Obviously, $B$ is $\mu$-measurable for any discrete probability measure $\mu$, and hence $B$ is universally measurable.

On the other hand, $B$ is not Borel, because otherwise, as an uncountable set, it would contain a copy of the Cantor set (see \cite[Theorem~13.6]{kechris}) which would contradict its concentration around $Q$.

Since $f$ is the characteristic function of an open set in $\Ch_\H(K)$, we have $f\r_{\Ch_\H(K)}\in \Bof_1(\Ch_\H(K))$. On the other hand,  $f$ is not Borel on $L\times\{0\}$ because the $\sigma$-algebra of Borel sets in $L$ coincides with the $\sigma$-algebra generated by resolvable sets in $L$ (see \cite[Proposition~3.4]{spurny-ahm}). Thus $f$ is the required function.
\end{proof}

\end{document}